\documentclass[12pt]{article}
\usepackage{amscd, amssymb, latexsym, amsmath}
\newtheorem{thm}{Theorem}
\newtheorem{cor}{Corollary}
\newtheorem{pro}{Proposition}

\newtheorem{rem}{Remark}
\newenvironment{proof}
{\noindent {\em Proof. }} {\hfill $\Box$}
\newcommand{\R}{\mathbb R}
\newcommand{\no}{\noindent}
\author{Kuo-Wei Lee* and Yng-Ing Lee**}
\date{}
\title{Mean Curvature flow in Higher Co-dimension}
\begin{document}
\maketitle \maketitle \leftline{*Department of Mathematics, National
Taiwan University, Taipei, Taiwan} \centerline{email:
d93221007@ntu.edu.tw} \leftline{**Department of Mathematics and
Taida Institute of Mathematical Sciences,} \leftline{\quad  National
Taiwan University, Taipei, Taiwan} \leftline{\quad  National Center
for Theoretical Sciences, Taipei Office}
 \centerline{email: yilee@math.ntu.edu.tw}
 \begin{abstract}
 We make several improvements on the results of M.-T. Wang
 in \cite{W1} and his joint paper with M.-P. Tsui \cite{TW1} concerning
 the long time existence and convergence for solutions of mean curvature flow in
 higher co-dimension. Both the curvature condition and
lower bound of $*\Omega$ are weakened.  New applications are
 also obtained.
  \end{abstract}
  \section{Introduction}
  From the first variation formula of area for a submanifold in
  a Riemannian manifold, we can consider the mean curvature vector as the
  negative gradient of the area functional. The area of the
  submanifold will decrease most rapidly if we deform the
  submanifold in the direction of its mean curvature vector. Such
  a deformation is called mean curvature flow. It is a very nature
  way to find minimal submanifolds, or canonical representatives.
  The study of mean curvature flow/curve shortening flow is very
  active and has much advance in the past thirty years. It started
  from the work of Brakke \cite{Br} and the  paper of
  Huisken \cite{Hu1} opened a new era on the mean curvature flow of
  hypersurface. New developments were obtained in recent years on
  mean curvature flow in higher co-dimension. Since our work mainly
  focuses on generalizing the results in \cite{TW1} and \cite{W1}, we
  do not intend to list all important developments and papers on
  mean curvature flow here.  Please refer to the papers \cite{TW1,W1}
  and the reference therein.

   In this note, we prove the following theorems:

   \no {\bf Theorem \ref{thm1} } {\it
Let $(N_1, g)$ and $( N_2, h)$ be two compact Riemannian
manifolds, and $f$ be a smooth map from
    $N_1$ to $N_2$.
Assume that $K_{N_1}\geq k_1$ and $K_{ N_2}\leq k_2$ for two
constants $k_1$ and $k_2$,
    where $K_{N_1}$ and $K_{ N_2}$ are the sectional curvature of $N_1$ and $N_2$ respectively.
Suppose either $k_1\geq 0, k_2\leq 0$, or $k_1\geq k_2>0$, then
the following results hold:

\begin{itemize}
\item[(i)] If $\frac{\det((g+f^*h)_{ij})}{\det(g_{ij})}<4$, then
the mean curvature flow of the graph of $f$ remains a graph of a
map and exists for all time. \item[(ii)] Furthermore, if $k_1>0$,
then the mean curvature flow converges smoothly to the graph of a
constant map.
\end{itemize}}

 \no {\bf Theorem \ref{thm2} } {\it Assume the same conditions as
 in Theorem \ref{thm1}. Then
the following results hold:
\begin{itemize}
\item[(i)] If $f$ is a smooth area decreasing map from $N_1$ to
$N_2$,
    then the mean curvature flow of the graph of $f$ remains the graph of an area decreasing map,
    and exists for all time.
\item[(ii)] Furthermore, if $k_1>0$, then the mean curvature flow
converges smoothly to the graph of a constant map.
\end{itemize}}

  In Theorem \ref{thm1} and Theorem \ref{thm2} , we generalize the curvature conditions on $N_1$ and $N_2$ of
  the main theorems in \cite{W1} and \cite{TW1} from constant sectional curvature to varied ones.
  Moreover, in Theorem \ref{thm1} the upper bound on
  $ \frac{\det((g+f^*h)_{ij})}{\det(g_{ij})}$ in \cite{W1} is relaxed from $2$
  to $4$, which should also be observed from \cite{TW1}. But
  since it is not mentioned and
  proved there,  for completeness we treat this generalization
as well. We also want to remark that the correct condition (which
is related to $*\Omega$) in \cite{W1} should be
$\frac{\sqrt{\det(g_{ij})}}{\sqrt{\det((g+f^*h)_{ij})}}$, instead
of $\frac1{\sqrt{\det((g+f^*h)_{ij})}}$.

We can apply Theorem \ref{thm2} to  show

\no {\bf Corollary \ref{cor1} }{\it Let $N_1, N_2$ be compact
manifolds and $\mbox{dim }N_1\geq 2$. Suppose that there exist
Riemannian metrics $g_1$ and $g_2$ on $N_1$ and $N_2$  with
sectional curvature $K_{N_1(g_1)}>0$ and $K_{N_2(g_2)}\leq 0$.
Then any map from $N_1$ to $N_2$ must be homotopic to a constant
map.}

\no {\bf Corollary \ref{App} }{\it Let $(N_1,g_1), (N_2,g_2)$ be
compact Riemannian manifolds with $K_{N_1(g_1)}\geq k_1$,
$K_{N_2(g_2)}\leq k_2$, and both $k_1$ and $k_2$ are positive
constants. If the $2$-dilation of $f:(N_1,g_1)\to(N_2,g_2)$ is
less than $\displaystyle\frac{k_1}{k_2}$, then $f$ is homotopic to
a constant map.}

 We made most of the observations in this paper a few years ago and explained
  the arguments to M.-T. Wang and M.-P. Tsui in 2004
  when the second author visited them in Columbia University. We
  thank M.-T. Wang for suggesting us to write up this note. A
  version of Theorem \ref{thm2} in pseudo-Riemannian case is
  obtained recently in \cite{LS}.

To prove Theorem 1 and 2, we  first need to show that the solution
of mean curvature flow remains the graph of a map satisfying the
same constraint as the initial map. This step depends on the
curvature condition. Once we obtain the inequality in the first
step, similar argument as in \cite{W1} shows that the solution
exists for all time. A refined inequality is needed to show that
$*\Omega$ will converge to 1 as $t$ tends to infinity. We also
need the curvature condition in this part. The last step, which is
to show that the limit is a graph of a constant map, is the same
as in \cite{W1}.

We list basic definitions and properties in \S 2 as preliminaries.
Theorem \ref{thm1} is proved in \S 3 and for completeness we also
sketch the argument for the part which is similar to \cite{W1}. In
\S 4, we discuss the area decreasing case and prove Theorem
\ref{thm2}. The applications are given in \S 5.

  \section{Preliminaries}
Assume that $N_1$ and $N_2$ are two compact Riemannian manifolds
with metric $g$ and $h$, and of dimension $n$ and $m$
respectively. Let $f:N_1\to N_2$ be a smooth map and denote  the
graph by $\Sigma$.
    Then $\Sigma$ is an embedded submanifold in the product manifold $M=N_1\times N_2$
    with
    $F=\mbox{id.}\times f:N_1\to M $.

    A smooth family $F_t:N_1\to M$ is called a mean curvature flow of $\Sigma$ if it satisfies
\[\begin{cases}\begin{aligned}&
\left(\frac{\partial F_t(x)}{\partial t}\right)^\bot=H(x,t) \\
&F_0(N_1)=\Sigma
\end{aligned}\end{cases}\]
where $H$ is the mean curvature vector of $F_t(N_1)=\Sigma_t$ and
$(\cdot)^\bot$ denotes the projection onto the normal bundle
    $N\Sigma_t$ of $\Sigma_t$. By standard theories, the flow has
    short time existence.

Let $\Omega$ be a parallel $n$-form on $M$. We can evaluate this
$n$-form on $\Sigma_t$. Choose orthonormal frames $\{e_i\}_{i=1}^n$
for $T\Sigma_t$ and $\{e_\alpha\}_{\alpha=n+1}^{n+m}$ on
$N\Sigma_t$. The following evolution equation for $\Omega$ is
derived by M.-T. Wang :

\begin{pro}\cite{W1}
If $F_t$ is an $n$-dimensional mean curvature flow of $\Sigma$ in
$M$ and $\Omega$ is a parallel $n$-form on $M$. Then
$\Omega_{1\cdots n}=\Omega(e_1,\ldots,e_n)$ satisfies
\begin{align}
\frac{\partial}{\partial t}\Omega_{1\cdots n}
=&\Delta\Omega_{1\cdots n}+\Omega_{1\cdots n}\left(\sum_{\alpha,i,k}(h_{ik}^\alpha)^2\right)\notag \\
-&2\sum_{\alpha<\beta,k}\left(\Omega_{\alpha\beta3\cdots
n}h_{1k}^\alpha h_{2k}^\beta
    +\Omega_{\alpha2\beta\cdots n}h_{1k}^\alpha h_{3k}^\beta+\cdots
    +\Omega_{1\cdots(n-2)\alpha\beta}h_{(n-1)k}^\alpha h_{nk}^\beta\right)\notag \\
-&\sum_{\alpha,k}\left(\Omega_{\alpha2\cdots n}R_{\alpha
kk1}+\cdots
    +\Omega_{1\cdots(n-1)\alpha}R_{\alpha kkn}\right) \label{evoeq0}
\end{align}
where $\Delta$ denotes the time-dependent Laplacian on $\Sigma_t$,
    $h_{ij}^\alpha=\langle\nabla_{e_i}^M e_j,e_\alpha\rangle$ is the second fundamental form,
    and $R$ is the curvature tensor of $M=N_1\times N_2$ with the product metric $g+h$.
\end{pro}

\begin{rem}
Here we use the same convention as in \cite{W1} that
\begin{align*}
&R(X,Y)Z=-\nabla_X\nabla_YZ+\nabla_Y\nabla_XZ+\nabla_{[X,Y]}Z \\
&R_{ijkl}=\langle R(e_k,e_l)e_i,e_j\rangle
\end{align*}
and the sectional curvature is $K(e_k,e_i)=\langle
R(e_k,e_i)e_k,e_i\rangle$, where $\{e_i\}$ are orthonormal.
\end{rem}

Since $M=N_1\times N_2$ is a product manifold, the volume form
$\Omega_1$ of
    $N_1$ can be extended as a parallel $n$-form on $M$.
At any point $p$ on $\Sigma_t$, we have
$*\Omega=\Omega_1(e_1,\ldots,e_n)=\Omega_1(\pi_1(e_1),\ldots,\pi_1(e_n))$,
which
    is the Jacobian of the projection from $T_p\Sigma_t$ to $T_{\pi_1(p)}N_1$.
By the implicit function theorem, we know $*\Omega>0$ near $p$ if
and only if $\Sigma_t$
    is locally a graph over $N_1$ near $p$.

When $\Sigma_t$ is the graph of $f_t:N_1\to N_2$, by the singular
value decomposition theorem,
    there exist an orthonormal basis $\{a_i\}_{i=1}^n$ for
    $T_{\pi_1(p)}N_1$ and $\{a_\alpha\}_{\alpha=n+1}^{n+m}$
for $T_{\pi_2(p)} N_2$ so that $df_t(a_i)=\lambda_ia_{n+i}$ for
$1\leq i\leq r$,
    and $df_t(a_i)=0$ for $r\leq i\leq n$.
Note that $r\leq\min(n,m)$ is the rank of $df_t$ at $p$, and
$\lambda_i$s' are the eigenvalues of
    $\sqrt{(df_t)^Tdf_t}$.
    Hence $\lambda_i\geq 0$ for all $i=1,\ldots,n$.
We can use $\{a_i\}_{i=1}^n$ and $\{a_\alpha\}_{\alpha=n+1}^{n+m}$
to construct special orthonormal bases
    $\{E_i\}_{i=1}^n$ on $T_p\Sigma_t$ and $\{E_\alpha\}_{\alpha=n+1}^{n+m}$
    on $N_p\Sigma_t$ as follows:
\begin{align}
E_i=&\left\{
\begin{array}{ll}
    \frac1{\sqrt{1+\lambda_i^2}}(a_i+\lambda_ia_{n+i}) & \mbox{ \ if \ } 1\leq i\leq r  \\
    a_i                                                & \mbox{ \ if \ } r+1\leq i\leq n
\end{array}\right. \label{onf1}\\
E_{n+q}=&\left\{
\begin{array}{ll}
    \frac1{\sqrt{1+\lambda_q^2}}(a_{n+q}-\lambda_qa_q) & \mbox{ if \ } 1\leq q\leq r    \\
    a_{n+q}                                            & \mbox{ if \ } r+1\leq q\leq m
\end{array}\right. \label{onf2}
\end{align}
Thus,
\begin{align*}
*\Omega=\Omega_1(\pi_1(E_1),\ldots,\pi_1(E_n))=\frac1{\sqrt{\prod_{i=1}^n(1+\lambda_i^2)}}.
\end{align*}

With these new bases (\ref{onf1}) and (\ref{onf2}), we can rewrite
(\ref{evoeq0}) as follows. This evolution equation is derived in
\cite{W1} and here we express the formula in a general form.

\begin{pro}\cite{W1}
Suppose $M=N_1\times N_2$ with the product metric $g+h$ and
$\Omega$ is the parallel extension of the volume form of $N_1$.
Let $\Sigma$ be an embedded submanifold in $M$ and be a graph over
$N_1$. If the mean curvature flow of $\Sigma$ is a graph over
$N_1$,
    then $*\Omega$ satisfies the following equation:
\begin{align}
&\frac{\partial}{\partial t}*\!\Omega
=\Delta*\!\Omega+*\Omega|A|^2+*\Omega
    \left\{2\sum_{k,i<j}\lambda_i\lambda_jh_{ik}^{n+j}h_{jk}^{n+i}
    -2\sum_{k,i<j}\lambda_i\lambda_jh_{ik}^{n+i}h_{jk}^{n+j}\right\}\notag\\
&+*\Omega\sum_{i,k}
    \left(\frac{\lambda_i^2}{(1+\lambda_i^2)(1+\lambda_k^2)}\langle R_1(a_k,a_i)a_k,a_i\rangle
    -\frac{\lambda_i^2\lambda_k^2}{(1+\lambda_i^2)(1+\lambda_k^2)}
    \langle R_2(a_{n+k},a_{n+i})a_{n+k},a_{n+i}\rangle\right) \label{evoeq1}
\end{align}
where $|A|^2$ denotes the norm square of the second fundamental
form, and
    $R_1, R_2$ denote the curvature tensors on $N_1,  N_2$ with metric $g, h$ respectively.
\end{pro}
\begin{proof}
From the evolution equation (\ref{evoeq0}) and bases (\ref{onf1}),
(\ref{onf2}), one has
\begin{align*}
&\Omega_{1\cdots\alpha\cdots\beta\cdots n}h_{ik}^\alpha
h_{jk}^\beta
    =*\Omega\lambda_i\lambda_j\left(h_{ik}^{n+i}h_{jk}^{n+j}-h_{jk}^{n+i}h_{ik}^{n+j}\right) \\
&\Omega_{1\cdots\alpha\cdots n}=-*\!\Omega\lambda_i \\
&R_{(n+i)kki}=\frac{-\lambda_i\lambda_k^2}{\left(1+\lambda_i^2\right)\left(1+\lambda_k^2\right)}
    \langle R_2(a_{n+k},a_{n+i})a_{n+i},a_{n+k}\rangle
    +\frac{\lambda_i}{\left(1+\lambda_i^2\right)\left(1+\lambda_k^2\right)}\langle R_1(a_k,a_i)a_i,a_k\rangle
\end{align*}
The evolution equation (\ref{evoeq1}) thus follows directly.
\end{proof}

When $*\Omega>0$, one can consider the evolution equation of
$\ln*\Omega$ instead and have the following:

\begin{pro}\cite{TW1} The evolution equation (\ref{evoeq1}) can be rewritten as the form:
\begin{align}
&\frac{\partial}{\partial t}\ln *\Omega =\Delta\ln
*\Omega+|A|^2+\sum_{i,k}\lambda_i^2\left(h_{ik}^{n+i}\right)^2
    +2\sum_{k,i<j}\lambda_i\lambda_jh_{ik}^{n+j}h_{jk}^{n+i}\notag \\
&+\sum_{i,k}\left(\frac{\lambda_i^2}{\left(1+\lambda_i^2\right)\left(1+\lambda_k^2\right)}
    \langle R_1(a_k,a_i)a_k,a_i\rangle-\frac{\lambda_i^2\lambda_k^2}{\left(1+\lambda_i^2\right)
    \left(1+\lambda_k^2\right)}\langle R_2(a_{n+k}, a_{n+i})a_{n+k}, a_{n+i}\rangle\right) \label{evoeq2}
\end{align}
\end{pro}

\begin{proof}
Since $\frac{\partial}{\partial
t}\ln*\Omega=\frac1{*\Omega}\left(\frac{\partial}{\partial
t}*\!\Omega\right)$,
   it implies $\frac{\partial}{\partial t}*\Omega=*\Omega\left(\frac{\partial}{\partial t}\ln*\Omega\right)$.
Similarly, one has
\begin{align*}
\Delta\ln*\Omega
=\frac{\Delta*\!\Omega}{*\Omega}-\frac{|\nabla*\!\Omega|^2}{|*\!\Omega|^2}
=\frac{\Delta*\!\Omega}{*\Omega}-\frac{|\Omega_{1\cdots
n,k}|^2}{|*\!\Omega|^2}
=\frac{\Delta*\!\Omega}{*\Omega}-\left|\sum_{i,k}\lambda_ih_{ik}^{n+i}\right|^2
\end{align*}
or
\begin{align*}
\Delta*\!\Omega=*\Omega(\Delta\ln*\Omega)+*\Omega\left(\sum_{i,k}\lambda_ih_{ik}^{n+i}\right)^2.
\end{align*}
Plugging these expressions into equation (\ref{evoeq1}) and
dividing $*\Omega$ on both sides,
   the equation (\ref{evoeq2}) is  then obtained.
\end{proof}
\section{Proof of Theorem 1}\label{section3}
Now we are ready to prove
\begin{thm} \label{thm1}
Let $(N_1, g)$ and $( N_2, h)$ be two compact Riemannian
manifolds, and $f$ be a smooth map from
    $N_1$ to $N_2$.
Assume that $K_{N_1}\geq k_1$ and $K_{ N_2}\leq k_2$ for two
constants $k_1$ and $k_2$,
    where $K_{N_1}$ and $K_{ N_2}$ are the sectional curvature of $N_1$ and $N_2$ respectively.
Suppose either $k_1\geq 0, k_2\leq 0$, or $k_1\geq k_2>0$, then
the following results hold:

\begin{itemize}
\item[(i)] If $\frac{\det((g+f^*h)_{ij})}{\det(g_{ij})}<4$, then
the mean curvature flow of the graph of $f$ remains a graph of a
map and exists for all time. \item[(ii)] Furthermore, if $k_1>0$,
then the mean curvature flow converges smoothly to the graph of a
constant map.
\end{itemize}
\end{thm}
{\it Proof of (i): }
 For convenience, we write equation
(\ref{evoeq2}) as
    \begin{equation}\label{evoeq3}\frac{\partial}{\partial t}\ln*\Omega=\Delta\ln*\Omega+\mbox{I}+\mbox{II},\end{equation}
     where
\begin{align*}
\mbox{I}
=&\mbox{ second fundamental form terms}\\
=&|A|^2+\sum_{i,k}\lambda_i^2\left(h_{ik}^{n+i}\right)^2+2\sum_{k,i<j}\lambda_i\lambda_jh_{ik}^{n+j}h_{jk}^{n+i} \\
\mbox{II}
=&\mbox{ curvature tensor terms} \\
=&\sum_{i,k}\left(\frac{\lambda_i^2}{\left(1+\lambda_i^2\right)
    \left(1+\lambda_k^2\right)}\langle R_1(a_k,a_i)a_k,a_i\rangle
    -\frac{\lambda_i^2\lambda_k^2}{\left(1+\lambda_i^2\right)\left(1+\lambda_k^2\right)}
    \langle R_2(a_{n+k}, a_{n+i})a_{n+k}, a_{n+i}\rangle\right) \\
=&\sum_{i,k\ne
i}\left(\frac{\lambda_i^2}{\left(1+\lambda_i^2\right)\left(1+\lambda_k^2\right)}K_{N_1}(a_k,a_i)
    -\frac{\lambda_i^2\lambda_k^2}{\left(1+\lambda_i^2\right)
    \left(1+\lambda_k^2\right)}K_{ N_2}(a_{n+k}, a_{n+i})\right)
\end{align*}
If we can show there exists $\delta>0$ such that
\begin{equation}\label{keyeq}\frac{\partial}{\partial
t}\ln*\Omega\geq\Delta\ln*\Omega+\delta|A|^2, \end{equation}
    by the maximum principle (the minimum version), $\min_{\Sigma_t}\ln*\Omega$ is nondecreasing in $t$,
    and $*\Omega\geq \min_{\Sigma_{t=0}}*\Omega>0$.
Thus $\Sigma_t$ remains the graph of a map $f_t:N_1\to N_2$
    whenever the flow exists. Moreover,
    since
\begin{equation}\label{Ombd}*\Omega=\frac{\sqrt{\det(g_{ij})}}{\sqrt{\det((g+f^*h)_{ij})}}
    =\frac1{\sqrt{\prod_{i=1}^n\left(1+\lambda_i^2\right)}},\end{equation}
    we have $\min_{\Sigma_{t=0}}*\Omega>\frac{1}{2}$, and thus $\min_{\Sigma_{t}}*\Omega>\frac{1}{2}$
    along the flow as well.

So we first  aim at proving equation (\ref{keyeq}). From
(\ref{Ombd})
  and the compactness of $N_1$, it follows that $\prod_{i=1}^n\left(1+\lambda_i^2\right)\leq 4-\varepsilon$
   on $\Sigma_{t=0}$ for some $\varepsilon>0$.  By continuity and the short time existence of the flow,
   the solution remains the graph of a map and satisfies
   $\prod_{i=1}^n\left(1+\lambda_i^2\right)\leq
   4-\frac{\varepsilon}{2}$ for small $t$.

In particular, when $i\neq j$,
$\left(1+\lambda_i^2\right)(1+\lambda_j^2)\leq
4-\frac{\varepsilon}{2}$. By mean inequality, we have
$|\lambda_i\lambda_j|\leq 1-\delta$ for
    $\delta=\frac\varepsilon8>0, i\ne j$.
Thus
\begin{align}
\mbox{I} \geq&\delta|A|^2+(1-\delta)\sum_{i,j,k}
    \left(h_{jk}^{n+i}\right)^2-2(1-\delta)\sum_{k,i<j}\left|h_{jk}^{n+i}h_{ik}^{n+j}\right| \notag \\
\geq&\delta|A|^2+(1-\delta)\sum_{k,i<j}\left(\left|h_{jk}^{n+i}\right|-\left|h_{ik}^{n+j}\right|\right)^2 \notag \\
\geq&\delta|A|^2 \label{ineqa}
\end{align}
For curvature tensor terms,
\begin{itemize}
\item[(a)] \underline{If $k_1\geq 0, k_2\leq 0$}, we have
\begin{align*}
\mbox{II}\geq
 \sum_{i,k\ne i}\left(\frac{\lambda_i^2}{\left(1+\lambda_i^2\right)\left(1+\lambda_k^2\right)}k_1
 -\frac{\lambda_i^2\lambda_k^2}{\left(1+\lambda_i^2\right)\left(1+\lambda_k^2\right)}k_2\right)\geq 0
\end{align*}
\item[(b)] \underline{If $k_1\geq k_2>0$}, then
\begin{align*}
\mbox{II} \geq&\sum_{i,k\ne
i}\left(\frac{\lambda_i^2}{\left(1+\lambda_i^2\right)\left(1+\lambda_k^2\right)}k_1
    -\frac{\lambda_i^2\lambda_k^2}{\left(1+\lambda_i^2\right)\left(1+\lambda_k^2\right)}k_2\right)\\
\geq&\sum_{i,k\ne
i}\left(\frac{\lambda_i^2-\lambda_i^2\lambda_k^2}{\left(1+\lambda_i^2\right)
    \left(1+\lambda_k^2\right)}\right)k_2
    =\sum_{i<k}\left(\frac{\lambda_i^2+\lambda_k^2-2\lambda_i^2\lambda_k^2}{\left(1+\lambda_i^2\right)
    \left(1+\lambda_k^2\right)}\right)k_2
\end{align*}
Since $|\lambda_i\lambda_k|<1$,
\begin{align*}
\lambda_i^2+\lambda_k^2-2\lambda_i^2\lambda_k^2
=(\lambda_i-\lambda_k)^2+2\lambda_i\lambda_k-2\lambda_i^2\lambda_k^2
=(\lambda_i-\lambda_k)^2+2\lambda_i\lambda_k(1-\lambda_i\lambda_k)\geq
0
\end{align*}
\end{itemize}
Hence $\mbox{II}\geq 0$.

Therefore (\ref{keyeq}) holds for small $t$. It follows that in fact
$*\Omega\geq\min_{\Sigma_{t=0}}*\Omega>\frac1{\sqrt{4-\varepsilon}}$
for small $t$. Thus we can continue the same argument to conclude
that the solution remains the graph of a map and satisfies $
*\Omega\geq\min_{\Sigma_{t=0}}*\Omega>\frac1{\sqrt{4-\varepsilon}}$ whenever the
flow exists.

Then by choosing  $\displaystyle
u=\frac{\ln*\Omega-\ln\Omega_0+c}{-\ln\Omega_0+c}$ with $c>0$ to
replace $*\Omega$, the same proof as in \cite{W1} leads to the
long-time existence of the flow. The only thing needed in the
proof is equation (\ref{keyeq}).

The idea goes as follows: To detect a possible singularity, say
$(y_0,t_0)$,  one first isometrically embeds $M$ into $\R^N$ by
Nash theorem, and introduces the backward heat kernel from Huisken
\cite{Hu2}
\begin{align*}
\rho_{y_0,t_0}=\frac1{(4\pi(t_0-t))^\frac{n}2}\mbox{e}^{-\frac{|y-y_0|^2}{4(t_0-t)}}
\end{align*}
Direct computation and using equation (\ref{keyeq}) give
\begin{align}
\frac{d}{dt}\int_{\Sigma_t}(1-u)\rho_{y_0,t_0}d\mu_t\leq
C-\delta\int_{\Sigma_t}|A|^2\rho_{y_0,t_0}d\mu_t \label{omegarho}
\end{align}
for some $C>0$. Therefore, $\lim\limits_{t\to
t_0}\int_{\Sigma_t}(1-u)\rho_{y_0,t_0}d\mu_t$ exists. Consider the
parabolic dilation $D_\lambda$ at $(y_0,t_0)$, that is,
\begin{align*}
(y,t)\stackrel{D_\lambda}{\longmapsto}(\lambda(y-y_0),\lambda^2(t-t_0)),
\end{align*}
and set $s=\lambda^2(t-t_0)$. Denote the corresponding submanifold
and volume form after dilation by $\Sigma^\lambda_s$ and
$d\mu^\lambda_s$ respectively. Because $u$ is invariant under
parabolic dilation, inequality (\ref{omegarho}) becomes
\begin{align}
\frac{d}{ds}\int_{\Sigma^\lambda_s}(1-u)\rho_{0,0}d\mu^\lambda_s
\leq\frac{C}{\lambda^2}-\delta\int_{\Sigma^\lambda_s}\rho_{0,0}|A|^2d\mu^\lambda_s
\label{scalerho}\end{align} With further discussion from
(\ref{scalerho}), one can find
 $\lambda_j\to\infty$ and  $s_j\to -1$ such that
 \begin{align}
\int_{\Sigma^{\lambda_j}_{s_j}\cap
K}|A|^2d\mu^{\lambda_j}_{s_j}\to 0 \mbox{ as } j\to\infty
\label{A}
\end{align}
for any compact set $K$. One can conclude that
$\Sigma^{\lambda_j}_{s_j}\to\Sigma^\infty_{-1}$ as Radon measure
and
    $\Sigma^\infty_{-1}$ is the graph of a linear function with further investigation.
Therefore,
\begin{align*}
\lim_{t\to
t_0}\int\rho_{y_0,t_0}d\mu_t=\lim_{j\to\infty}\int\rho_{0,0}d\mu^{\lambda_{j}}_{s_j}=1
\end{align*}
It implies that $(y_0,t_0)$ is a regular point by White's theorem
in \cite{Wh}, which is a contradiction. Thus no singularity can
occur along the flow. We refer to \cite{W1} for the detailed
argument.

 \hfill $\Box$

\no {\it Proof of (ii): } We use the same expression as in
(\ref{evoeq3}) and will first show that there exists $c_0>0$ which
depends on $\varepsilon, k_1,n$ such that
\begin{align*}
\mbox{II}\geq c_0\sum_{i=1}^n\lambda_i^2\geq
c_0\ln\left(\prod_{i=1}^n\left(1+\lambda_i^2\right)\right)=
-2c_0\ln*\Omega .
\end{align*}
\begin{itemize}
\item[(a)] \underline{If $k_1>0$, and $k_2\leq 0$}, we have
\begin{align*}
\mbox{II} \geq&\sum_{i,k\ne
i}\left(\frac{\lambda_i^2}{\left(1+\lambda_i^2\right)\left(1+\lambda_k^2\right)}k_1
    -\frac{\lambda_i^2\lambda_k^2}{\left(1+\lambda_i^2\right)\left(1+\lambda_k^2\right)}k_2\right) \\
\geq&\sum_{i,k\ne
i}\frac{\lambda_i^2k_1}{\left(1+\lambda_i^2\right)\left(1+\lambda_k^2\right)}\\
\geq&\frac{k_1(n-1)}{4}\sum_{i=1}^n\lambda_i^2
\\ \geq&
\frac{k_1(n-1)}{4}\sum_{i=1}^n\ln (1+\lambda_i^2)
\end{align*}
since $\displaystyle
\frac{1}{\left(1+\lambda_i^2\right)\left(1+\lambda_k^2\right)}\geq
\frac1{\prod_{i=1}^n\left(1+\lambda_i^2\right)}\geq\frac14$ and
$\lambda_i^2 \geq \ln (1+\lambda_i^2)$. Hence we can take
$\displaystyle c_0=\frac{k_1(n-1)}4$.

\item[(b)] \underline{If $k_1\geq k_2>0$}, we need to estimate
curvature terms more carefully. Recall
\begin{align*}
\mbox{II} \geq&\sum_{i,k\ne
i}\left(\frac{\lambda_i^2}{\left(1+\lambda_i^2\right)\left(1+\lambda_k^2\right)}k_1
    -\frac{\lambda_i^2\lambda_k^2}{\left(1+\lambda_i^2\right)\left(1+\lambda_k^2\right)}k_2\right)\\
\geq&\sum_{i,k\ne
i}\left(\frac{\lambda_i^2-\lambda_i^2\lambda_k^2}{\left(1+\lambda_i^2\right)
    \left(1+\lambda_k^2\right)}\right)k_1
    =\sum_{i<k}\left(\frac{\lambda_i^2+\lambda_k^2-2\lambda_i^2\lambda_k^2}{\left(1+\lambda_i^2\right)
    \left(1+\lambda_k^2\right)}\right)k_1
\end{align*}
As observed in the proof of (i), we have
$|\lambda_i\lambda_k|<1-\frac\varepsilon4$ for all $t\geq 0$.
Thus,
\begin{align*}
\lambda_i^2+\lambda_k^2-2\lambda_i^2\lambda_k^2
=\lambda_i\lambda_k(\lambda_i-\lambda_k)^2+(1-\lambda_i\lambda_k)(\lambda_i^2+\lambda_k^2)
\geq \frac\varepsilon4(\lambda_i^2+\lambda_k^2)
\end{align*}
Therefore,
\begin{align*}
\mbox{II}\geq\frac{\varepsilon
k_1}{16}\sum_{i<k}(\lambda_i^2+\lambda_k^2) =\frac{\varepsilon
k_1(n-1)}{16}\sum_{i=1}^n\lambda_i^2\geq \frac{\varepsilon
k_1(n-1)}{16}\sum_{i=1}^n\ln(1+\lambda_i^2)
\end{align*}
We can take $\displaystyle c_0=\frac{\varepsilon k_1(n-1)}{16}$.
\end{itemize}
Hence we can rewrite  (\ref{evoeq3}) as
\begin{align}\label{evoeq4}
\frac{\partial}{\partial
t}\ln*\Omega\geq\Delta\ln*\Omega-2c_0\ln*\Omega
\end{align}
Consider a function $f(t)$ which depends only on $t$ and satisfies
\begin{equation}\label{auxfun}\begin{cases}\begin{aligned}&
\frac{d}{dt}f(t)=-2c_0f(t)\\
&f(0)=\min\limits_{\Sigma_{t=0}}\ln*\Omega
\end{aligned}\end{cases}\end{equation}
which gives $f(t)=f(0)\mbox{e}^{-2c_0t}$. From the inequality
(\ref{evoeq4}) and (\ref{auxfun}),
    we have $$\frac{\partial}{\partial t}(\ln*\Omega-f(t))
    \geq\Delta(\ln*\Omega-f(t))-2c_0(\ln*\Omega-f(t)).$$
Because $\min_{\Sigma_{t=0}}(\ln*\Omega-f(t))\geq 0$,  by the
maximum principle, we have $\displaystyle
\min_{\Sigma_{t>0}}(\ln*\Omega-f(t))\geq 0$. Hence $\displaystyle
0\geq\ln*\Omega\geq f(0)\mbox{e}^{-2c_0t}$ on $\Sigma_{t\geq 0}$.
Letting $t\to\infty$, it gives $*\Omega\to 1$. Then one can apply
the same argument as in \cite{W1} to conclude that the solution
converges
    smoothly to a constant map at infinity.
We outline the proof for this fact in next paragraph.

Given $\varepsilon_1>0$,
    there exists $T$ such that $*\Omega>\frac1{\sqrt{1+\varepsilon_1}}$
    for $t>T$.
    It implies $\sum_i\lambda_i^2<\varepsilon_1$ for $t>T$.
The same method as in (\ref{ineqa}) and taking $\delta$ larger,
for example $\displaystyle \delta=\frac12$,
    gives
$ \displaystyle \frac{\partial}{\partial
t}*\!\Omega\geq\Delta*\!\Omega+\frac12*\!\Omega|A|^2$. The
evolution equation for the second fundamental form is
\begin{align*}
\frac{\partial}{\partial t}|A|^2\leq\Delta|A|^2-2|\nabla
A|^2+K_1|A|^4+K_2|A|^2
\end{align*}
for some constants $K_1, K_2$. The $K_1|A|^4$ term will cause some
trouble, but one can consider the evolution inequality of
$(*\Omega)^{-2p}|A|^2$, which is
     \begin{align*}
&\frac{\partial}{\partial t}\left((*\Omega)^{-2p}|A|^2\right) \\
\leq&\,\Delta\left((*\Omega)^{-2p}|A|^2\right)
-(*\Omega)^{-2p}\nabla\left((*\Omega)^{-2p}\right)\cdot\nabla\left((*\Omega)^{-2p}|A|^2\right)
\\& \,
+(*\Omega)^{-2p}|A|^2\left(|A|^2\left(K_1-p+2p(p-1)n\varepsilon_1\right)+K_2\right)
\end{align*}
    Choose $\varepsilon_1$ small, and a suitable $p=p(n,\varepsilon_1)$
     so that the coefficient of the highest order nonlinear term in the evolution inequality of
    $(*\Omega)^{-2p}|A|^2$ is negative.
By the maximum principle,  one gets an upper bound of
$\max_{\Sigma_t}|A|^2$ and concludes that $\max_{\Sigma_t}|A|^2\to
0$ as $t\to\infty$.
    It implies that the mean curvature flow of $\Sigma$ converges to a totally geodesic submanifold of $M$.
Since $*\Omega\to 1$ as $t\to\infty$, we have $|df_t|\to 0$ and
the limit is a constant map.

\hfill $\Box$
\begin{rem}
When $n=1$, then $k_1=0$ and (ii) cannot apply. In fact, term
$I\!I$ vanishes in this case and one cannot obtain the convergence
using the same method.
\end{rem}
\section{The area-decreasing case} In this section, we mainly
follow the discussion and set-up in \cite{TW1}. Consider a
parallel symmetric two tensor $S$ on $M$ defined as
\begin{align*}
S(X,Y)=g(\pi_1(X),\pi_1(Y))-h(\pi_2(X),\pi_2(Y)),
\end{align*}
where $\pi_1$ and $\pi_2$ are the projections into $TN_1$ and
$TN_2$ respectively. The same  calculation as for $*\Omega$ leads
to  the following evolution equation for  $S$ on $\Sigma_t$, which
appears in \cite{TW1},
\begin{align*}
\left(\frac{\partial}{\partial t}-\Delta\right)S_{ij}
=&-h_{il}^\alpha h_{kk}^\alpha S_{lj}-h_{jl}^\alpha h_{kk}^\alpha
S_{li}
    +R_{kik\alpha}S_{\alpha j}+R_{kjk\alpha}S_{\alpha i} \notag \\
    &+h_{kl}^\alpha h_{ki}^\alpha S_{lj}+h_{kl}^\alpha h_{kj}^\alpha S_{li}-2h_{ki}^\alpha h_{kj}^\beta S_{\alpha\beta}
\end{align*}
where $S_{ij}=S(e_i,e_j), S_{\alpha i}=S(e_\alpha, e_i),
S_{\alpha\beta}
    =S(e_\alpha,e_\beta), i,j=1,\ldots,n; \alpha,\beta=n+1,\ldots,n+m$.

One can simply the equation in terms of evolving orthonormal
frames. Denote $\bar{g}=g+h$ which is the product metric on
$M=N_1\times N_2$.   Suppose that
  $\bar{F}=\{F_1,\ldots,F_a,\ldots,F_n\}$ are orthonormal frames on $T_p\Sigma_t$.
We evolve $\bar{F}$ by the formula
\begin{align} \label{evframe}
\frac{\partial}{\partial t}F^i_a
=\bar{g}^{ij}\bar{g}_{\alpha\beta}h^\alpha_{kj}H^\beta F^k_a
\end{align}
where $\alpha$ and $\beta$ are in the normal direction  and
$H^\beta$ is the $\beta$ component of the mean curvature vector.

Let $S_{ab}=S_{ij}F^i_aF^j_b=S(F_a,F_b)$ be the component of
$S$ in $\bar{F}$. Then $S_{ab}$ satisfies the following equation
\begin{align}
\left(\frac{\partial}{\partial t}-\Delta\right)S_{ab}
=R_{cac\alpha}S_{\alpha b}+R_{cbc\alpha}S_{\alpha a}
+h^\alpha_{cd}h^\alpha_{ca}S_{db}+h^\alpha_{cd}h^\alpha_{cb}S_{da}-2h^\alpha_{ca}h^\beta_{cb}S_{\alpha\beta}
\label{onfs2}
\end{align}

We remark that when we use the bases (\ref{onf1}) and
(\ref{onf2}), the expression of $S$ is
\begin{align*}
S=S(E_i,E_j)_{1\leq i,j\leq n+m} =\left(\begin{array}{cccc}
B & 0 & D & 0\\
0 & I_{(n-r)\times(n-r)} & 0 & 0\\
D & 0 & -B & 0\\
0 & 0 & 0 & -I_{(m-r)\times(m-r)}
\end{array}\right)
\end{align*}
where $B$ and $D$ are $r$ by $r$ matrices with
\begin{align*}
B_{ij} =S(E_i,E_j)
=\frac{1-\lambda_i^2}{1+\lambda_i^2}\delta_{ij}\quad\mbox{and}\quad
D_{ij} =S(E_i,E_{n+j})
=-\frac{2\lambda_i}{1+\lambda_i^2}\delta_{ij}.
\end{align*}

A map $f:N_1\to N_2$ is called {\it area-decreasing} if
\begin{align*}
\left|\wedge^2df\right|(x)=\sup_{|u\wedge
v|=1}\left|\left(\wedge^2df\right)(u\wedge v)\right|
=\sup_{|u\wedge v|=1}|df(u)\wedge df(v)|<1.
\end{align*}
In the bases (\ref{onf1}) and (\ref{onf2}), the area-decreasing
condition is equivalent to
\begin{align*}
\left|\wedge^2df\right|(x)=\sup_{i<j}\lambda_i\lambda_j<1
\Leftrightarrow |\lambda_i\lambda_j|<1\ \forall \, i \neq j\, .
\end{align*}
On the other hand, the sum of any two eigenvalues of $S$ is
\begin{align*}
\frac{1-\lambda_i^2}{1+\lambda_i^2}+\frac{1-\lambda_j^2}{1+\lambda_j^2}
=\frac{2(1-\lambda_i^2\lambda_j^2)}{(1+\lambda_i^2)(1+\lambda_j^2)}
\end{align*}
Thus, the area-decreasing condition is equivalent to two
positivity of $S$ .

Since $S$ is bilinear, by the Riesz representation theorem,
    we can identify $S$ with a self-adjoint operator(still denoted by $S$).
Hence, for the orthonormal frame $\bar{F}$, we have
$S_{ab}=S(F_a,F_b)=\bar{g}(S(F_a),F_b)$, which implies
$S(F_a)=S_{ab}F_b$.

With this identification, we can construct a new self-adjoint
operator
    $S^{[2]}=S\otimes 1+1\otimes S$ on $T_p\Sigma_t\wedge T_p\Sigma_t$,
    which is defined by $ S^{[2]}(w_1\wedge w_2)=S(w_1)\wedge w_2+w_1\wedge
    S(w_2)$.
     If $\mu_1\leq\cdots\leq\mu_n$ are the eigenvalues of $S$ with
the corresponding eigenvectors $v_1,\ldots,v_n$,
    then $S^{[2]}$ has eigenvalues $u_{i_1}+u_{i_2}$ with eigenvectors $v_{i_1}\wedge v_{i_2}, i_1\leq i_2$.
Thus, the positivity of $S^{[2]}$ is equivalent to the area
decreasing condition. Similarly, for the metric $\bar{g}$, we can
construct a self-adjoint operator $\bar{g}^{[2]}=\bar{g}\otimes
1+1\otimes\bar{g}$.

Note that  $\{F_a\wedge F_b\}_{a<b}$ form an orthonormal basis for
$\wedge^2T\Sigma_t$ and
\begin{align}
S^{[2]}(F_a\wedge F_b) =&S(F_a)\wedge F_b+F_a\wedge S(F_b)
=S_{ac}F_c\wedge F_b+F_a\wedge S_{ac}F_c \notag \\
=&\sum_{c<d}(S_{ac}\delta_{bd}+S_{bd}\delta_{ac}-S_{ad}\delta_{bc}-S_{bc}\delta_{ad})F_c\wedge F_d \label{S2} \\
\bar{g}^{[2]}(F_a\wedge
F_b)=&\sum_{c<d}(2\delta_{ac}\delta_{bd}-2\delta_{ad}\delta_{bc})F_c\wedge
F_d \notag
\end{align}

We can improve the main theorem in \cite{TW1} to the following
\begin{thm} \label{thm2}
Let $(N_1, g)$ and $(N_2, h)$ be two compact Riemannian manifolds,
and $f$ be a smooth map from $N_1$ to $N_2$. Assume that
$K_{N_1}\geq k_1$ and $K_{N_2}\leq k_2$ for two constants $k_1$
and $k_2$,
    where $K_{N_1}$ and $K_{N_2}$ are the sectional curvature of $N_1$ and $N_2$ respectively.
Suppose either $k_1\geq 0, k_2\leq 0$, or $k_1\geq k_2>0$, then the following results hold:
\begin{itemize}
\item[(i)] If $f$ is a smooth area decreasing map from $N_1$ to $N_2$,
    then the mean curvature flow of the graph of $f$ remains the graph of an area decreasing map,
    and exists for all time.
\item[(ii)] Furthermore, if $k_1>0$, then the mean curvature flow converges smoothly to the graph of a constant map.
\end{itemize}
\end{thm}
\begin{proof}
Notice that we already  prove in section \ref{section3} that
$\Sigma _t$ remains the graph of a map under the assumption
whenever the flow exists.
   Now we want to prove that the area-decreasing property is also
   preserved along the mean curvature flow.
 Since the initial map is area-decreasing, there exists
$\varepsilon>0$ such that $S^{[2]}-\varepsilon \bar{g}^{[2]}\ge
0$. We want to show that the property $S^{[2]}-\varepsilon
\bar{g}^{[2]}$ is preserved along the mean curvature flow. Let
$M_\eta=S^{[2]}-\varepsilon \bar{g}^{[2]}+\eta t\bar{g}^{[2]}$.
Suppose the mean curvature flow exists on $[0,T)$. Consider any
$T_1<T$, it suffices to show that $M_\eta>0$ on $[0,T_1]$ for all
$\eta<\frac{\varepsilon}{2 T_1}$. If it does not hold, there will
be a first time $0<t_0<T_1$, where $M_\eta$ is nonnegative
definite,
    and there is a null eigenvector $V=V^{ab}F_a\wedge F_b$ for $M_\eta$ at some point $x_0\in\Sigma_{t_0}$.
We extend $V$ to a parallel vector field in a neighborhood of $x_0$ along geodesic emanating out of $x_0$,
    and defined $V$ on $[0,T)$ independent of $t$.

Define a function $f=M_\eta(V,V)$, then the function $f$ has the
following properties at $(x_0,t_0)$:
\begin{itemize}
\item[(F1)] $f=0$ \hspace{22mm}($V$ is the null-eigenvector)
\item[(F2)] $\nabla f=0$ \hspace{18mm}(At $t=t_0$, $f$ attains
minimum on $x_0$) \item[(F3)] $\left(\frac{\partial}{\partial
t}-\Delta\right)f\leq 0$ \hspace{5mm}(At $t=t_0$, $f$ attains
minimum on $x_0$)
\end{itemize}

At $(x_0,t_0)$, we choose the orthonormal basis $\{F_a\}$ as
$\{E_i\}$ in (\ref{onf1}),
    and rearrange them such that the singular values $\lambda_i$ satisfy
    $\lambda_1\geq \lambda_2\geq\cdots\geq \lambda_n\geq 0$.
Thus,
\begin{align*}
S_{nn}=\frac{1-\lambda_n^2}{1+\lambda_n^2}\geq\cdots\geq S_{22}
=\frac{1-\lambda_2^2}{1+\lambda_2^2}\geq S_{11}=\frac{1-\lambda_1^2}{1+\lambda_1^2}
\end{align*}
Hence the null eigenvector must be $V=E_1\wedge E_2$. From (F1),
it follows that $f=S_{11}+S_{22}+2(\eta t_0-\varepsilon)=0$ at
$(x_0,t_0)$ which implies $S_{11}+S_{22}=2(\varepsilon-\eta
t_0)>0$. Thus, we have
\begin{align}
\lambda_1\lambda_2<1, \quad\mbox{and} \quad\lambda_i<1
\quad\mbox{for}\quad i\geq 2 \label{lambda}
\end{align}

Use (\ref{evframe}) to evolve $\{F_a\}$. Then at $(x_0,t_0)$,
direct computation gives
\begin{align*}
\left(\frac{\partial}{\partial t}-\Delta\right)f
=&2\eta+2R_{k1k\alpha}S_{\alpha 1}+2R_{k2k\alpha}S_{\alpha 2} \\
+&2h^\alpha_{kj}h^\alpha_{k1}S_{j1}+2h^\alpha_{kj}h^\alpha_{k2}S_{j2}
    -2h^\alpha_{k1}h^\beta_{k1}S_{\alpha\beta}-2h^\alpha_{k2}h^\beta_{k2}S_{\alpha\beta} \\
=&2\eta+\mbox{I}+\mbox{II}
\end{align*}
where
\begin{align}
\mbox{I}
=&\mbox{curvature tensor terms} \notag \\
=&2R_{k1k\alpha}S_{\alpha 1}+2R_{k2k\alpha}S_{\alpha 2}=2R_{k1k(n+1)}S_{(n+1)1}+2R_{k2k(n+2)}S_{(n+2)2} \notag \\
=&\sum_{k\neq 1}\frac{2\lambda_1^2}{(1+\lambda_k^2)(1+\lambda_1^2)^2}\langle R_1(a_k,a_1)a_k, a_1\rangle
    +\sum_{k\neq 2}\frac{2\lambda_2^2}{(1+\lambda_k^2)(1+\lambda_2^2)^2}\langle R_1(a_k,a_2)a_k,a_2\rangle \notag \\
-&\sum_{k\neq 1}\frac{2\lambda_k^2\lambda_1^2}{(1+\lambda_k^2)(1+\lambda_1^2)^2}\langle R_2(a_k,a_1)a_k,a_1\rangle
    -\sum_{k\neq 2}\frac{2\lambda_k^2\lambda_2^2}{(1+\lambda_k^2)(1+\lambda_2^2)^2}
    \langle R_2(a_k,a_2)a_k,a_2\rangle \notag \\
\geq&\sum_{k\neq 1}\frac{2\lambda_1^2}{(1+\lambda_k^2)(1+\lambda_1^2)^2}k_1
    +\sum_{k\neq 2}\frac{2\lambda_2^2}{(1+\lambda_k^2)(1+\lambda_2^2)^2}k_1 \notag \\
-&\sum_{k\neq 1}\frac{2\lambda_k^2\lambda_1^2}{(1+\lambda_k^2)(1+\lambda_1^2)^2}k_2
    -\sum_{k\neq 2}\frac{2\lambda_k^2\lambda_2^2}{(1+\lambda_k^2)(1+\lambda_2^2)^2}k_2 \label{ctta1} \notag \\
\mbox{II}
=&\mbox{second fundamental form terms} \notag \\
=&2h^\alpha_{kj}h^\alpha_{k1}S_{j1}+2h^\alpha_{kj}h^\alpha_{k2}S_{j2}
    -2h^\alpha_{k1}h^\beta_{k1}S_{\alpha\beta}-2h^\alpha_{k2}h^\beta_{k2}S_{\alpha\beta} \notag
\end{align}

For curvature tensor terms I,
\begin{itemize}
\item[(a)] \underline{If $k_1\geq 0, k_2\leq 0$}, we have $\mbox{I}\geq 0$.
\item[(b)] \underline{If $k_1\ge k_2>0$}, then
\begin{align*}
\mbox{I}
&\geq k_1\left(\sum_{k\neq 1}\frac{2\lambda_1^2-2\lambda_k^2\lambda_1^2}{(1+\lambda_k^2)(1+\lambda_1^2)^2}
    +\sum_{k\neq 2}\frac{2\lambda_2^2-2\lambda_k^2\lambda_2^2}{(1+\lambda_k^2)(1+\lambda_2^2)^2}\right) \\
&=k_1\left(\frac{2\lambda_1^2-2\lambda_2^2\lambda_1^2}{(1+\lambda_2^2)(1+\lambda_1^2)^2}
    +\frac{2\lambda_2^2-2\lambda_1^2\lambda_2^2}{(1+\lambda_1^2)(1+\lambda_2^2)^2}\right. \\
&\hspace*{12mm}\left.+\sum_{k\geq 3}\frac{2\lambda_1^2-2\lambda_k^2\lambda_1^2}{(1+\lambda_k^2)(1+\lambda_1^2)^2}
    +\sum_{k\geq 3}\frac{2\lambda_2^2-2\lambda_k^2\lambda_2^2}{(1+\lambda_k^2)(1+\lambda_2^2)^2}\right) \\
&\geq k_1\left(\frac{2\lambda_1^2+2\lambda_2^2-4\lambda_2^2\lambda_1^2}{(1+\lambda_1^2)^3}\right)
    +\sum_{k\geq 3}k_1\left(\frac{2\lambda_1^2(1-\lambda_k^2)}{(1+\lambda_k^2)(1+\lambda_1^2)^2}
    +\frac{2\lambda_2^2(1-\lambda_k^2)}{(1+\lambda_k^2)(1+\lambda_2^2)^2}\right) \\
&\geq k_1\left(\frac{2(\lambda_1-\lambda_2)^2+4\lambda_1\lambda_2(1-\lambda_1\lambda_2)}{(1+\lambda_1^2)^3}\right)
    \hspace{15mm}\mbox{(here we use (\ref{lambda}))} \\
&\geq 0 \hspace{79mm}\mbox{(here we use (\ref{lambda}))}
\end{align*}
\end{itemize}

Since the second fundamental form terms do not involve curvatures,
$\mbox{II}$ is nonnegative as proved in \cite{TW1}. Since  both
$\mbox{I}\geq 0$ and $\mbox{II}\geq 0$ at $(x_0,t_0)$,
   we have $\left(\frac{\partial}{\partial t}-\Delta\right)f\ge 2\eta>0$ at $(x_0,t_0)$,
    which contradicts to (F3).
Thus the area-decreasing property is preserved by the mean
curvature flow. We can also apply the same proof to obtain
long-time existence and convergence as in section \ref{section3}.
The theorem is therefore proved.
\end{proof}

\section{Application}
\begin{cor} \label{cor1}
Let $N_1, N_2$ be compact manifolds and $\mbox{dim }N_1\geq 2$.
Suppose that there exist Riemannian metrics $g_1$ and $g_2$ on
$N_1$ and $N_2$  with sectional curvature $K_{N_1(g_1)}>0$ and
$K_{N_2(g_2)}\leq 0$. Then any map from $N_1$ to $N_2$ must be
homotopic to a constant map.
\end{cor}
\begin{proof}
For any given map $f:N_1\to N_2$, we can consider the singular
value decomposition of $df$ with respect to $g_1$ and $g_2$.
Denote the corresponding singular values by
$\lambda_1,\ldots,\lambda_n$. Since $N_1$ is compact, there exists
a positive constant $L$ such that $\lambda_i\lambda_j\leq L$.
Define a new metric $\bar{g}_1=2Lg_1$ on $N_1$. The singular
values of $df$ with respect to $\bar{g}_1$ and $g_2$ will be
$\bar{\lambda}_1=\frac{\lambda_1}{\sqrt{2L}},\ldots,\bar{\lambda}_n=\frac{\lambda_n}{\sqrt{2L}}$.
 Therefore, we have $\bar{\lambda}_i\bar{\lambda}_j\leq
\frac12<1$ and $K_{N_1(\bar{g}_1)}>0$.  Applying the mean
curvature flow to the graph of $f$ in
$(N_1,\bar{g}_1)\times(N_2,g_2)$, by Theorem \ref{thm2} we
conclude that $f$ is homotopic to a constant map.
\end{proof}

For general cases, we can obtain the null homotopic property in
terms of $2$-dilation. Recall that the $2$-dilation (or more
generally, $k$-dilation) of a map $f$ between $N_1$ and $N_2$ is
said at most $D$ if $f$ maps each $2$-dimensional
($k$-dimensional) submanifold in $N_1$ with volume $V$ to an image
with volume at most $DV$. The $2$-dilation can also be defined in
terms of $df$, which is equal to the supremum of the norm
$|\wedge^2\!df|$.

We have the following corollary:
\begin{cor}\label{App}
Let $(N_1,g_1), (N_2,g_2)$ be compact Riemannian manifolds with
$K_{N_1(g_1)}\geq k_1$, $K_{N_2(g_2)}\leq k_2$, and both $k_1$ and
$k_2$ are positive constants. If the $2$-dilation of
$f:(N_1,g_1)\to(N_2,g_2)$ is less than
$\displaystyle\frac{k_1}{k_2}$, then $f$ is homotopic to a
constant map.
\end{cor}
\begin{proof}
Consider the metrics $\bar{g}_1=k_1g_1$ and $\bar{g}_2=k_2g_2$.
Then the sectional curvatures satisfy $K_{N_1(\bar{g}_1)}\geq 1,
K_{N_2(\bar{g}_2)}\leq 1$, and the map
$f:(N_1,\bar{g}_1)\to(N_2,\bar{g}_2)$ satisfies
$|\wedge^2\!df|<\frac{k_1}{k_2}\cdot\frac{k_2}{k_1}=1$, which is
an area-decreasing mapping. By Theorem \ref{thm2}, $f$ is
homotopic to a constant map.
\end{proof}

 Assume
$(N_1,g_1)$ has nonnegative Ricci curvature and $\mbox{dim
}N_1=2$. A classical result in harmonic theory tells us that there
exists $\varepsilon>0$ such that if a harmonic map $f:(N_1,g_1)\to
(N_2,g_2)$ satisfies $E(f)=\int_{N_1}\|df\|^2<\varepsilon$, then
$f$ is a constant map. As a final application of Theorem
\ref{thm2}, one can prove a similar result. The idea is first to
obtain the pointwise bound of $df$ by the total energy. Then apply
Corollary \ref{App} to conclude that $f$ is homotopic to a
constant map when $K(g_1)>0$. Such a pointwise estimate is
obtained by Schoen \cite{S} when $\mbox{dim }N_1=2$, $f$ is
harmonic, and the energy is sufficiently small in small balls. We
remark that this argument works in higher dimension whenever the
pointwise estimate is obtained.

\end{document}